\documentclass{amsart}
\usepackage{a4,fancyhdr,geometry}

\newtheorem{theorem}{Theorem}[section]
\theoremstyle{plain}

\newtheorem{definition}{Definition}
\newtheorem{example}{Example}

\newtheorem{lemma}{Lemma}[section]

\newtheorem{remark}{Remark}

\numberwithin{equation}{section}

\begin{document}
	\title[AL-monoids]{Autometrized Lattice Ordered monoids
	}
	\author[T. R. Ashale]{Tekalign Regasa Ashale} 
	\address{{\bf Tekalign Regasa Ashale}
		\\Department of Mathematics, Addis Ababa University, Addis Ababa, Ethiopia}
	\email{tekalign.regasa@aau.edu.et}

	\author[G. A. Abebe]{Girum AKlilu Abebe} 
	\address{{\bf Girum AKlilu Abebe}
		\\Department of Mathematics, Addis Ababa University, Addis Ababa, Ethiopia}
	\email{girum.aklilu@aau.edu.et}
	
	\author[Kolluru Venkateswarlu ]{Kolluru Venkateswarlu } 
	\address{{\bf Kolluru Venkateswarlu }
		\\Department of Computer Science and Systems Engineering, Andhra University, Visakhaptanam, India}
	\email{drkvenkateswarlu@gmail.com}

	\subjclass[2020]{06F05, 06B10, 20M14, 20M12}
	

	
	\setcounter{page}{1}
	
	\begin{abstract}
		In this paper, we introduce autometrized lattice ordered monoids (AL-monoids), a generalization of DRl-semigroups, and establish several algebraic properties of AL-monoids. We also investigate the isometries of AL-monoids, proving that the set of invertible elements of an AL-monoid forms a l-group. Finally, we show that an algebra $A$ with identity is a Boolean algebra if, for each $a \in A$, the mapping $x \mapsto a \ast x$ is an isometry. Throughout the chapter We use $A$ as AL-monoids and $a,b,c,d, x,y$ as elements of AL-monoids.
	\end{abstract}
	\keywords{DRl-semigroup, Autometrized algebra, Lattice ordered autometrized algebra, Representable autometrized algebra, AL-monod, Isometry}.
	\maketitle

	\section*{Introduction}
	In \cite{swamy1965dually}, Swamy initiated the study of DRl- semi groups (dually residuated lattice ordered semi groups) as an answer  to the question $''$ is there a common abstraction that includes Boolean algebras (rings) and l-groups as special cases$''$ posed by Birkhoff in his book \cite{birkhoff1973lattice}. However, there were several other  solutions for the common abstraction  of Birkhoff's problem. Clans by Wyler \cite{wyler1966clans}, multi-rings by Nakano \cite{nakano1967rings},  common abstraction by Rama Rao \cite{rao1969common,rao1972common}(he did not mention any name for the system, since it turned out to be the direct product of Boolean ring and l-group) are a few. In \cite{sw}, Swamy introduced the concept of autometrized algebras as a unified theory of Browerian algebras \cite{Nord}, commutative l-groups, Boolean -l-algebras \cite{Rao}. All these algebras (which are lattices)possess a metric namely symmetric difference. However there are some other algebras namely semi- brouwerian algebra \cite{Murty}, Newmann algebra \cite{Newm}  do possess a metric but they are not even lattices. 
	
	Subba Rao \cite{BV d} introduced and studied a special subclass namely representable autometrized algebras. He made a through study on the geometric and algebraic aspects of representable autometrized algebras. In 2023, M Melese \cite{tilahun2023structure;Structure of an autometrized algebra} study structure of autometrized algebra by introducing subalgebra and examines ideals, homomorphisms, quotient algebra, and isomorphism theorem within autometrized algebra.
	
	Subba Rao in \cite{BV d} introduced the notion of representable autometrized algebras fairly a wider class to the class of DRl-semigroups and extensively studied algebraic and geometric aspects of these classes. Further he made a study on lattice ordered autometrized algebras which is semiregular autometrized algebra, and all the mappings are contractions concerning semiregular operations, and showed that it possesses the geometry shared by Boolean algebras and lattice ordered groups, DRl- semigroups, etc. While looking at the generalization, some of the properties may not hold.
	For instance any representable autometrized algebra is not equationally definable and also not distributive.
	
	To fill this kind of gap, we  introduce the concept of Autometrized Lattice Ordered Monoid (AL-monoid), which is a generalization of a  dually residuated lattice-ordered semigroup (DRl-semigroup)and a subclass of representable autometrized algebra. The class of AL-monoids enjoy the properties of a variety. Furthermore, we investigate the algebraic properties and isometries in AL-monoids.
	
	In \textbf{section 2}, we recall the definitions of a DRl-semi group  and an Autometrized algebra from \cite{rao2019note: a note on representable autometrized algebra, subba2018metric} and \cite{ swamy1965dually, swamy1965duallyb}.
	
	In \textbf{section 3}, we introduce the  notion of an  AL-monoid and furnish examples. Further we give an example of an AL-monoid which is not a DRl-semi group establishing that AL-monoid is a generalization of DRl-semi group. We establish certain algebraic consequences of of AL-monoid. Also, we show that the set of all complemented elements form a Boolean algebra and the set of all invertible elements form a l-group. we conclude the section by showing, every AL-monoid with 1 form a Boolean algebra if the map $x\mapsto a\ast x$ is an isometry. 
	
	\section{Preliminaries}
	\begin{definition}\cite{swamy1965dually}
		An Autometrized algebra A is a system $(A,+,\leq,\ast )$ where 
		\begin{itemize} 
			\item[{(a)}] $(A,+)$ is a binary commutative algebra with element 0.
			\item[{(b)}] $\leq$ is antisymmetric, reflexive ordering on A.
			\item[{(c)}] $\ast : A\times A\rightarrow A$ is a mapping satisfying the formal properties of distance, namely,
			\begin{itemize}
				\item[{(1)}] $a\ast b\geq 0$ for all $a,b$ in A,equality,if and only if $a=b$.
				\item[{(2)}]$a\ast b=b\ast a$ for all $a,b$ in A, and
				\item[{(3)}] $a\ast b\leq a\ast c +c\ast b$ for all $a,b,c$ in A.
			\end{itemize}
		\end{itemize}
	\end{definition}
	\begin{definition}
		\cite{swamy1965duallyb} A system $A=(A,+,\leq,\ast )$ of arity $(2,2,2)$ is called a lattice ordered autometrized algebra, if and only if, A satisfies the following conditions.
		\begin{itemize}
			\item[{(1).}] $(A,+,\leq)$ is a commutative lattice ordered semi-group with $'0'$, and
			\item[{(2).}] $\ast$ is a metric operation on A.i.e, $\ast$ is a mapping from $A\times A$ into A satisfying the formal properties of distance, namely,
			\begin{itemize}
				\item[{(a)}] $a\ast b\geq 0$ for all $a,b$ in A,equality,if and only if $a=b$,
				\item[{(b)}]$a\ast b=b\ast a$ for all $a,b$ in A, and
				\item[{(c)}] $a\ast b\leq a\ast c +c\ast b$ for all $a,b,c$ in A.
			\end{itemize}
		\end{itemize}
	\end{definition}
	\begin{definition}
		\cite{rao2019note: a note on representable autometrized algebra} A lattice ordered autometrized algebra $A=(A,+,\leq,\ast )$of arity $(2,2,2)$ is called representable autometrized algebra, if and only if, A satisfies the following conditions:
		\begin{itemize}
			\item[{(1)}] $A=(A,+,\leq,\ast)$ is semiregular autometrized algebra. Which means $a\in A$ and $a\geq 0$ implies $a\ast 0=a,$ and 
			\item[{(2)}]  For every a in A, all the mappings $x\mapsto a+x,x\mapsto a\vee x,x\mapsto a\wedge x$ and $x\mapsto a\ast x$ are contractions (i.e.,if $\theta$ denotes any one of the operations $+,\wedge,\vee$and $\ast$,then,for each a in A,$(a\theta x)\ast (a\theta y)\leq x\ast y$ for all $x,y$ in A)
		\end{itemize}
	\end{definition}
	\begin{definition}\cite{sw}
		An autometrized algebra $(A,+,\leq, \ast )$ is called normal autometrized algebra if it satisfies:
		\begin{enumerate}
			\item $a\leq a\ast 0$ for all $a$ in $A.$
			\item $(a+c)\ast (b+d)\leq (a\ast b)+(c\ast d)$ for all $a, b, c, d$ in $A.$
			\item $(a\ast c)\ast (b\ast d)\leq (a\ast b)\ast (c\ast d),$ for all $a,b,c,d \in A.$ 
			\item For all $a,b$ in $A,$ if $a\leq b,$ then there exists $x\geq 0$ such that $a+x=b.$
		\end{enumerate}
	\end{definition}
	\begin{definition}\cite{Normal}
		An algebra $(A, +,\wedge, \vee, \ast, 0)$ is a normal autometrized lattice ordered algebra (NAl-algebra) if the following holds.
		\begin{itemize}
			\item[(i).] $(A,+,\leq)$ is an abelian lattice ordered monoid
			\item[(ii).] $\ast$ is metric operation.
			\item[(iii).] $x \ast 0\geq x$ for each $x\in A.$
			\item[(iv).] $(x+y)\ast (x'+y')\leq (x\ast x' +y\ast y')$ for all $x,y,x',y' \in A.$
			\item[(v).] $(x\ast y)\ast (x'\ast y') x,y,x',y' \in A$
			\item[(vi).] $x,y\in A$ and $x\leq y$ imply there exist $z\in A$ such that $x+z=y.$
		\end{itemize}
	\end{definition}
	\begin{definition}\cite{swamy1965dually} \label{d}
		Dually residuated lattice ordered semigroup (DRl-semigroup) is an algebra $A=(A,+,\leq,-,0)$ of type $(2,2,2,0)$ satisfying
		\begin{itemize}
			\item[{(1)}] $(A,+,\leq,0)$ is a commutative lattice ordered semigroup with identity element 0.
			\item[{(2)} ]for every element $a,b\in A$ there is a least element $x$ such that $x+b\geq a$ for which x is uniquely determined as $a-b.$
			\item[{(3)}] $(a-b)\vee 0+b\leq a\vee b.$
			\item[{(4)}] $a-a\geq 0.$
		\end{itemize}
	\end{definition}
	
	We begin with the following definition:
	\begin{definition}\label{2.1}
		An autometrized lattice ordered monoid (AL-monoid) is an algebra $(A,+,\vee,\wedge,\ast,0)$ of arity $(2,2,2,2,0)$ where:
		\begin{enumerate}
			\item $(A,+,\vee,\wedge,0)$ is a commutative lattice ordered monoid with least element $0$.
			\item $a \ast (a \wedge b) + b = a \vee b$ for all $a, b\in A.$
			\item The mappings $x \mapsto a+x$, $a \vee x$, $a \wedge x$, and $a \ast x$ are contractions with respect to $\ast$ (i.e., a mapping $f: A \to A$ is a contraction with respect to $\ast$ if $f(x) \ast f(y) \leq x \ast y$, where $\leq$ is the ordering in $A$ induced by $(A,\vee,\wedge)$).
			\item $[a \ast (a \vee b)] \wedge [b \ast (a \vee b)] = 0.$
		\end{enumerate}
	\end{definition}
	\begin{remark}\label{rem1}
		\begin{enumerate}
			\item The conditions $1,2$ and $3$ yield a common abstraction of commutative l-groups and Brouwerian algebras.
			\item All conditions in the above definition yield a common abstraction of commutative l-groups and Boolean algebras.
		\end{enumerate}
	\end{remark}

	\begin{example} Every Boolean algebra $(B,\vee,\wedge,^{\prime},0;,1)$ is an AL-monoid with $a \ast b = (a \wedge b^{\prime}) \vee (a^{\prime} \wedge b)$.
	\end{example}
	\begin{example} Every commutative l-group $(G,\vee,\wedge,+,-)$, with $- $ which is the operation inverse of $+$, is an AL-monoid if we define $a \ast b = (a + (-b)) \vee (b + (-a))$, where $-a$ and $-b$ represent the additive inverses of $a$ and $b$ respectively.
	\end{example}
	\begin{remark} Every DRl-semigroup is an AL-monoid since it is the common abstraction of Brouwerian algebras and commutative l-groups. It satisfies remark \ref{rem1}.; however, the converse does not hold. The following Example \ref{ex1}  illustrates an AL-monoid that is not a DRl-semigroup.
	\end{remark}
	\begin{example} \label{ex1}Let $A = \mathbb{Z^{+}} \cup \{u\}$, where $\mathbb{Z^{+}}$ is the set of all integers and $u$ is an element not in $\mathbb{Z}$. For all $a, b \in \mathbb{Z}$, define the operations as follows: 
		- $a + b$ is the usual sum, 
		- $a + u = u = u + a$, 
		- $u + u = u$, 
		- $a \ast b = |a - b|$, 
		- $a \ast u = u = u \ast a$. 
		$a\leq u, u\wedge u=u$ and $u\vee u=u, u\ast u=0.$
		Define $\leq$ in $A$ for all $a \in \mathbb{Z}$ as the usual ordering. Then $A = (A, \leq, +, \ast)$ is an AL-monoid that is not a DRl-semigroup since there is no unique least element in $A$ such that $x + u \geq a$ for any $a$ in $\mathbb{Z}$.
	\end{example}
	\begin{example} Normal autometrized lattice ordered algebras are AL-monoids. But the converse is not true. See Example \ref{ex1}.\end{example}
	\begin{remark}The following examples demonstrate that conditions (2) and (4) of Definition \ref{2.1} are independent in any AL-monoid.
	\end{remark}
	\begin{example}\label{ex2.9} Let $A$ be the lattice of all closed subsets of the real numbers with the usual topology. It can be verified that $A$ satisfies for all $a, b\in A; a \ast (a \wedge b) + (a \wedge b) = a$, but does not satisfy condition (4). For instance. Consider $x = [0,2]$ and $y = [2,3]$. Then $(x \ast (x \vee y)) \wedge (y \ast (x \vee y)) = \{2\} \neq 0.$ 
	\end{example}
	\begin{remark}
		From Example \ref{ex2.9}, it can be observed that the class of representable autometrized algebras is broader than the class of AL-monoids.
	\end{remark}
	\begin{example}	
		Let $A = \mathbb{Z} \vee \{u,v\}$, where $\mathbb{Z}$ is the set of all integers, and $u$ and $v$ are elements not in $\mathbb{Z}$. Define:
		$a + b$ as the usual sum,
		$a + u = u = u + a$,
		$a + v = v = v + a$ for all $a \in \mathbb{Z}$,
		$u + v = u = v + u$, 
		$u + u = u$, 
		$v + v = v$.
		Define $\leq$ in $A$ as follows: for elements in $\mathbb{Z}$, let $\leq$ be the usual ordering; define $u < a < v$ for all $a \in \mathbb{Z}$. 
		Define $\ast$ in $A$ as follows: 
		$a \ast b = |a - b|$ for all $a, b \in \mathbb{Z}$,
		$a \ast u = v = u \ast a$, 
		$u \ast v = v = v \ast u$, 
		$u \ast y = 0 = v \ast v$. 
		It can be easily verified that $(A,+,\leq,\ast)$ satisfies axioms (1), (3), and (4) of Definition \ref{2.1}, but does not satisfy axiom (2) because $v \ast (v \wedge u) + (v \wedge u) = v \ast u + (v \wedge u) = v + u \neq v.$
	\end{example}
	\begin{remark}
		Axiom (ii) of Definition \ref{2.1} shows AL-monoid is  equationally definable, ensure that the class of AL-monoids is closed under the formation of subalgebras, direct unions, and homomorphic images, thus forming a variety.
	\end{remark}
	\section{Algebraic properties}\label{s21}
	\begin{lemma}
		$a=a\ast (a\wedge b)+(a\wedge b) $ for all $a, b$ in $A$.
		\begin{proof}
			$a=a\vee(a\wedge b)~~\text{...absorption law}=a\ast (a\wedge (a\wedge b)) +(a\wedge b)...\text{from (2)~~of definition \ref{2.1}}=a\ast (a\wedge b)) +(a\wedge b)$
		\end{proof}
	\end{lemma}
	
	\begin{lemma}\label{L2.1}
		$a \ast (a \wedge b) = (a \vee b) \ast b$ for all $a, b\in A.$
		\begin{proof}
			Let $a, b \in A$. Then, $a \ast (a \wedge b) = ((a \vee b) \wedge a) \ast (a \wedge b)~~\text{....absoption law}=((a\vee b)\wedge a)\ast ((a\wedge b)\wedge a)) \leq (a \vee b) \ast (a \wedge b).\text{~~by (iii) of Definition \ref{2.1}}. $ and, $ b \ast (a \vee b) = ((a \wedge b) \vee b) \ast (a \vee b) = (b \vee (a \wedge b)) \ast (b \vee a) \leq a \ast (a \wedge b) \text{ (by axiom (3) of Definition \ref{2.1})}.$
		\end{proof}
	\end{lemma}
	\begin{lemma}\label{L2.2}
		Let $A$ be an AL-monoid and $a,b \in A.$ Then $ a\ast 0=0.$ (regularity). It follows that $\ast$ is semiregular.
	\end{lemma}
	\begin{proof}
		We have $a = a \ast (a \wedge 0) + (a \wedge 0) = a \ast 0$. In particular, $0 \ast 0 = 0$.
	\end{proof}
	\begin{lemma}Let $A$ be an AL-monoid and $a\in A.$ Then \( a \ast a = 0 \)\end{lemma} \label{lem1} \begin{proof}
		$a\vee a=a\ast (a\wedge a)+a \implies a=a\ast a+a \implies a\ast a=0.$
	\end{proof}		
	\begin{lemma} Let $A$ be an AL-monoid and $a,b \in A.$ Then \( a \ast b \geq 0 \).\end{lemma} \begin{proof}Consider the following:
		$0 = 0 \ast 0 = [0 \ast (0 \wedge (a \wedge b)) + (0 \wedge (a \wedge b))] \ast [0 \ast (0 \wedge (a \wedge b)) + (0 \wedge (a \wedge b))]\leq (0 \wedge (a \wedge b)) \ast (0 \wedge (a \wedge b))= 0 \wedge (a \wedge b \wedge a) \ast (0 \wedge a \wedge b \wedge b)\leq (a \wedge b) \ast (a \wedge b) \leq a \ast b.$ Thus, we conclude that \( a \ast b \geq 0 \).	Since \( a \ast b \geq 0 \), we have \( a \ast a = 0 \).
	\end{proof}
	\begin{lemma}Let $A$ be an AL-monoid and $a,b \in A.$ Then \( a \ast b = b \ast a \).\end{lemma}
	\begin{proof} $a \ast b = (a \ast (a \wedge b) + (a \wedge b)) \ast (b \ast (a \wedge b) + (a \wedge b)) \leq (a \ast (a \wedge b)) \ast (b \ast (a \wedge b))\leq (a \wedge b) \ast b \ast (a \wedge b) \leq b \ast a.$ Interchanging \( a \) and \( b \) gives \( b \ast a \leq a \ast b \).\end{proof}
	\begin{lemma} Let $A$ be an AL-monoids and $a,b \in A.$ Then $a\ast b=0 \implies a=b.$ \end{lemma}
	\begin{proof}
		To show \( a \ast b = 0 \) implies \( a = b \): $a = a \ast (a \wedge b) + (a \wedge b) = ((a \wedge a) \ast (a \wedge b)) + (a \wedge b) \leq (a \ast b) + (a \wedge b) \leq (a \ast b) + b = 0 + b = b.$
		Thus, \( a \ast b = 0 \) implies \( b \ast a = 0 \) and hence \( b \leq a \). Therefore, \( a = b \). 
	\end{proof}
	\begin{lemma}: $a\ast c\leq a\ast b+a\ast c;$ for
		all $a,b,c$ in $A.$ \end{lemma}
	\begin{proof}	$a \ast b = (a \ast b) \ast ((a \ast b) \wedge (c \ast b)) + (a \ast b) \wedge (c \ast b) \leq (a \ast b) \ast (c \ast b) + (c \ast b) \leq (a \ast c) + (c \ast b).$
	\end{proof}
	\begin{theorem}
		An AL-monoid \( A \) is an autometrized algebra.
		\begin{proof}
			Follows from (i) of Definition \ref{2.1} and lemma 
			Thus, \( A \) is an autometrized algebra.
		\end{proof}
	\end{theorem}
	\begin{theorem}
		Every AL-monoids are representable autometrized algebra.\end{theorem}
	\begin{proof}
		This follows from Lemma \ref{L2.2} and condition (3) of Definition \ref{2.1}.
	\end{proof}
	\begin{example}
		A Brouwerian algebra is representable autometrized algebra. But not AL-monoids. Since there is no least element in AL-monoids $A$ such that $a+x\geq a.$
	\end{example}
	\begin{theorem}\label{T2:3}
		Let $A$ be any AL-monoids and $a, b, c \in A$. The following conditions hold:
		\begin{itemize}
			\item[(i)] $b \leq a \Rightarrow a = a \ast b + b$,
			\item[(ii)] $a \vee b = a \ast b + a \wedge b$,
			\item[(iii)] $a \ast b = (a \vee b) \ast (a \wedge b)$,
			\item[(iv)] $a \ast b = (a \ast (a \wedge b)) + (a \wedge b) \ast b = a \ast (a \vee b) + (a \vee b) \ast b$, for all $a, b \in A$.
		\end{itemize}
	\end{theorem}
	
	\begin{proof}
		$a=a\vee(a\wedge b)=a\ast (a\wedge (a\wedge b))+(a\wedge b)\text{~~~~from (ii) of defn \ref{2.1}}= a\ast (a\wedge b)+(a\wedge b). $ Thus	(i).	$b \leq a \implies b=a\wedge b$. Then putting into equation equation(3.2)  we have $a = a \ast b + b.$
		
		For (ii), we used (iii) of Definition \ref{2.1} consider:
		
		$a\leq a\vee b;~~~ b\leq (a\vee b)...\text{(A, $\leq$ )~~is a lattice).}$
		$a\vee b\leq (a\vee b)\ast b \leq a\ast b+ b~~~\text{by (i). } \& (a\vee b)\ast a \leq b\ast a=a\ast b+a.$ This implies $a\vee b\leq (a\ast b)\wedge (a\ast b+a)=a\ast b+a\wedge b.$ For the reverse, we have: $(a \ast b) + (a \wedge b) = (a \ast (a \wedge b) + (a \wedge b)) \ast (b \ast (b \wedge a) + (b \wedge a)) + a \wedge b\leq ((a \ast (a \wedge b)) \ast (b \ast (b \wedge a)) + (a \wedge b))\text{(iii) of Definition \ref{2.1}} = ((a  \ast (a\wedge b)+b=a\vee b.$ Thus, $a \vee b = a \ast b + a \wedge b$.
		
		To prove (iii), let: $ a \ast b = (a \ast (a \wedge b) + a \wedge b) \ast (b \ast (a \wedge b) + (a \wedge b)) \leq (a \ast (a \wedge b)) \ast (b \ast (a \wedge b)) \text{ (by axiom (3) of Definition \ref{2.1})} \leq  (b \ast (a \vee b)) \ast (b \ast (a \wedge b)) \text{ (by Lemma \ref{L2.1})} = (a \vee b) \ast (a \wedge b). $
		
		Finally, for (iv), let $a, b \in A$. By the triangle inequality, we have:
		$ a \ast b \leq a \ast (a \wedge b) + (a \wedge b) \ast b \leq (a \vee b) \ast b + b \ast (a \wedge b) \text{ (by Lemma \ref{L2.1})} = (a \ast b + a \wedge b) \ast (b \ast (a \wedge b) + (a \wedge b)) + b \ast (a \wedge b) \text{ (by (2) of this theorem)} \leq (a \ast b) \ast (b \ast (a \wedge b)) + (b \ast (a \wedge b)) \text{ (by axiom (3) of Definition \ref{2.1})}\leq ((a \ast b) \ast 0) =a \ast b \text{ (by (i) of this theorem, since $b \ast (a \wedge b) \leq b \ast a = a \ast b$)}.$
		
		Now we have $a \ast b = a \ast (a \wedge b) + (a \wedge b) \ast b = (a \vee b) \ast b + (a \vee b) \ast a \text{ (by Lemma \ref{L2.1})} = a \ast (a \vee b) + (a \vee b) \ast b.$
	\end{proof}
	\begin{remark}
		Not all Lattice ordered autometrized algebra is representable.
	\end{remark}
	\begin{example}
		Not all lattice ordered autometrized algebras are AL-monoids. Similarly Brouwerian algebra are lattice ordered autometrized algebra. But, not AL-monoids.
	\end{example}
	\begin{remark}
		Not all Normal autometrized algebras are AL-monoids.
	\end{remark}
	\section{Isometries}\label{s22}
	
	\begin{theorem}
		Let $a, b \in A$. Then the following statements hold:
		\begin{enumerate}
			\item $a \wedge 0$ is invertible.
			\item If $a$ is invertible, then $a \wedge b$ is invertible.
			\item If $a$ is invertible, then $a \vee 0$ is invertible.
			\item If $a$ and $b$ are invertible, then $a \wedge b$ is invertible.
		\end{enumerate}
		\begin{proof}
			1. To show that $a \wedge 0$ is invertible, note that:$(0 \ast (0 \wedge a) + (0 \wedge a))=0\ast (0\wedge (a\wedge 0))+(a\wedge 0)=0\vee (0\wedge a) = 0.$Thus, $a \wedge 0$ is invertible.
			
			2. If $a$ is invertible, then $a + x = 0$ implies:
			$a \ast (a \wedge (a \wedge b)) + (a \wedge (a \wedge b)) + x =0 \Rightarrow (a \ast (a \wedge (a \wedge b)) + x) + (a \wedge (a \wedge b)) = 0.$
			This shows that $a \wedge b$ is invertible.
			
			3. Since both $a$ and $a \wedge 0$ are invertible, it follows that $a \vee 0$ is invertible as $a = a \vee 0 + a \wedge 0$.
			
			4. Finally, for invertible $a$ and $b$, we have:
			$a \vee b = a \ast (a \wedge b) + b  \implies a =a \ast (a \wedge (a \wedge b) + (a \wedge b)= a \vee (a \wedge b) = a.$	As $a$ and $a \wedge b$ are invertible, it follows that $a \ast (a \wedge b)$ is also invertible, proving (4).
		\end{proof}
	\end{theorem}
	
	
		
		
		\begin{lemma}
			The inverse in AL-monoid is unique.
		\end{lemma}
		\begin{proof}
			Let \( a, b, c \in A \) be elements of the AL-monoid. Suppose \( b \) and \( c \) are both inverses of \( a \). This means  $ b + a = 0 $ and $c + a = 0$. From the first equation, we have $b = -a.$ From the second equation, it follows that $c = -a$ Since both \( b \) and \( c \) are equal to \( -a \), we conclude that $b = c.$Thus, the inverse of \( a \) is unique. \end{proof}
		\begin{theorem}\label{T17}
			The set of all invertible elements of an AL-monoid $A$ forms a l-group.
		\end{theorem}
		\begin{proof}
			If $a$ and $b$ are invertible and $a\leq b,$ then $a+(-b)\leq 0.$ Which implies that $-a+(-b)\leq -a.$ Thus, $-b\leq -a$ and hence the $-(a\vee b)=-a\wedge -b-(a\wedge b)=-a\vee -b.$ Therefore, by this relation the class becomes l-group.
		\end{proof}
		
		
	
	\textbf{Notation:} We write $-b$ for the inverse of an invertible element $b$ and $a - b$ for $a + (-b)$.
	\begin{lemma}
		If $a, b$ are invertible, then so is $a\ast b.$ 
	\end{lemma}
	\begin{proof}
		$a\vee b=a\ast(a\wedge b)+b))\implies a\ast b+b=a\vee (a\wedge b)=a.$ 	\end{proof} Hence $a\ast b$ is invertible.
	\begin{theorem}
		If $a, b$ are invertible elements of $A$, then:
		\begin{enumerate}
			\item[(i)] $a \ast (a \wedge b) = (a - b) \vee 0$;
			\item[(ii)] $a \ast b = (a \vee b) - (a \wedge b)$;
			\item[(iii)] $a \ast 0 \geq a \vee 0 \geq a$; and
			\item[(iv)] If $x$ is invertible and $a + x = b + x$, then $a = b$.
		\end{enumerate}
		\begin{proof}
			(i). We have $a \ast (a\wedge b), a \wedge b$ as elements of a l-group. Now,
			$	a \ast (a \wedge b) + a \wedge b =  a \implies a \ast (a \wedge b) = a - (a \wedge b) = a + (-a \vee -b)= (a + (-a)) \vee (a + (-b)) = (a - a) \vee (a - b) = 0 \vee (a - b) = (a - b) \vee 0.$
			Hence (i) holds.
			
			(ii). Since $x \geq y$ and $x, y$ are invertible, it implies $x - y \geq 0$. Therefore,
			$x \ast (x \wedge y) = (x - y) \vee 0 = x - y.$
			Thus,
			$a \ast b = (a \vee b) \ast (a \wedge b) = (a \vee b) - (a \wedge b).$
			Hence (ii) holds.
			
			(iii). By (4) of Definition \ref{2.1}, we have:
			$a \vee 0 = a \ast (a \wedge 0) = (a \wedge a) \ast (a \wedge 0) \leq a \ast 0.$
			Thus, (iii) holds.
			
			(iv). If $x$ is invertible and $a + x = b + x$, then:
			$(a + x) - x =(b + x) - x \implies a + 0 = b + 0 \implies  a = b.$
			Hence (iv) holds.
		\end{proof}
	\end{theorem}
	
	\begin{definition}
		An element $a \in A$ is idempotent if $a + a = a$.
	\end{definition}
	
	\begin{definition}
		A bijective mapping $\sigma: A \to A$ is an isometry if $\sigma(x) \ast \sigma(y) = x \ast y$ for all $x, y \in A$.
	\end{definition}
	
	\begin{theorem}
		The mapping $x \mapsto a + x$ is an isometry for all $a$ and $x$ in $A$  if $a$ is invertible.
		\begin{proof}
			If $a$ is invertible, then $a + x = a + y$ implies $x = y$, showing that $x \mapsto a + x$ is injective.
			
			For $y \in A$, let $x = y - a$ so that $x + a = y$. Thus, it is a bijection. 
			
			Now, we check if it is an isometry:$(a \ast x) \ast (a \ast y) \leq x \ast y = (-a + a + x) \ast (-a + a + y) \leq (a + x) \ast (a + y) = x \ast y.$
			Hence, the map is an isometry.
			
			Conversely, if the map is an isometry, then $0 = a + x$ for some $x$ implies $a$ is invertible.
		\end{proof}
	\end{theorem}
	
	\begin{theorem}
		If $a, b$ are idempotents of an AL-monoid $A$, then:
		\begin{itemize}
			\item[(i)] $a \geq 0$.
			\item[(ii)] $a \wedge b$ is also idempotent.
			\item[(iii)] $a \vee b = a + b$.
			\item[(iv)] Both $a \vee b$ and $a + b$ are idempotents.
			\item[(v)] $a + b = a \vee b + a \wedge b$ is idempotent.
		\end{itemize}
	\end{theorem}
	
	\begin{proof}
		1. If $a + a = 0$, then:$ a + a \vee 0 + a \wedge 0 = a \vee 0 + a \wedge 0 \text{~~ 0 in lattice represent a least element}.$
		This implies $ a + a \vee 0 = (a + a) \vee (a + 0) = a \vee 0.$ Thus, $a \vee a = a \vee 0 $ 
		$\implies a = a \vee 0 \geq 0.$
		
		2. For $a \wedge b$, we have:$a \wedge b + a \wedge b = (a \wedge b + a) \wedge (a \wedge b + b) = (a + a) \wedge (a + b) \wedge (b + b)= a \wedge b \wedge (a + b).$ Hence (ii) holds.
		
		3. Since $a \ast b + a \wedge b = a \vee b$, we have:
		$a \ast b + a \wedge b + a \wedge b = a \vee b + a \wedge b \Longrightarrow a \ast b + a \wedge b = a \vee b + a \wedge b.$ Thus, $a \vee b = a + b$.
		
		4. If $(a + b) + (a + b) = (a + a) + (b + b)= a + b,$ then $a + b$ is idempotent. Also, since $a + b = a \vee b$ from (3), (iv) holds.
		
		5. By considering (iii), $a + b = a \vee b + a \wedge b$ is idempotent.
	\end{proof}
	
	\begin{theorem}
		For any $a, b \in A$, $a \ast b = (a \ast (a \wedge b)) \ast (b \ast (a \wedge b))$.
		\begin{proof}
			We have:
			$a \ast b = ((a \ast (a \wedge b)) \ast (b \ast (b \wedge a))) \leq (a \ast (a \wedge b) \ast (b \ast (a \wedge b)))$ and,$(a\ast (a\wedge b))\ast (b\ast (a\wedge b))\leq a\ast b~~ \text{by ~~ (iii)~~ of Definition \ref{2.1}}$. Hence, the proof.
		\end{proof}
	\end{theorem}
	\begin{theorem}
		Let $A = (A, +, \leq, \ast, 0)$ be an AL-monoid. Then $(A, \leq)$ is a distributive lattice.
		\begin{proof}
			Let $a,x,y,z, w$ are elements of $A$.	Clearly, $(A, \leq)$ is a lattice. Let $a \wedge x = a \wedge y$ and $a \vee x = a \vee y$. Then, by axiom 2 of Definition \ref{2.1} and Lemma \ref{L2.1}, we can drive the following equations:
			$(x \ast (a \wedge x)) + (a \wedge x) = x$
			$(y \ast (a \wedge y)) + (a \wedge y) = y.$
			
			Given that \( a \wedge x = a \wedge y \), we can denote this common value as \( z \) (i.e., \( z = a \wedge x = a \wedge y \)).
			Now, we can rewrite the previous equations as:
			$(x \ast z) + z = x	$ and $
			(y \ast z) + z = y.$		
			Next, we express \( x \) and \( y \) in terms of their joins with \( a \):$
			x = ((a \vee x) \ast a) + z$ and $
			y = ((a \vee y) \ast a) + z.$
			Since \( a \vee x = a \vee y \), we can denote this common join as \( w \) (i.e., \( w = a \vee x = a \vee y \)).
			
			We then have: $x = (w \ast a) + z \quad \text{and} \quad y = (w \ast a) + z.$
			Since both expressions for \( x \) and \( y \) are equal, it follows that \( x = y \).
			
			Therefore, we have shown that \( a \wedge x = a \wedge y \) and \( a \vee x = a \vee y \) implies \( x = y \). This demonstrates that the lattice is distributive.
		\end{proof}
	\end{theorem}
	
	\begin{remark}
		Let $A$ have unity, i.e., there exists $1$ such that $a + (a \ast 1) = 1.$ If $a \in A^{+}$, then $a \leq 1$ and $0 \leq a \ast 1 \Rightarrow a = a + 0 \leq a + (a \ast 1) = 1$.
	\end{remark}
	
	With usual terminology, we have the following:
	
	\begin{lemma}
		\begin{itemize}
			\item[(i).] If $a$ is complemented then $a+a=a.$
			\item[(ii).] $a\wedge (a\ast 1)=0$
			\item[(iii).] $a\vee (a\vee 1)=1.$
			\item[(iv).] $a'=a \ast 1$
		\end{itemize}
	\end{lemma}
	\begin{proof}
		(i).	Let $a$ have a complement $a'$ in $A$. Then,
		$a = a + 0 = a + (a \wedge a') = (a + a) \wedge (a + a')= (a + a) \wedge 1= a + a.$\\
		(ii). $(a\ast 1)+a=1=a'+a \implies a\ast 1 \leq a'$ (by (i)) implies $a\wedge (a\ast 1)\leq a\wedge a'=0.$\\
		(iii). $a\ast (a\wedge (a\ast 1))+(a\ast 1)=a\vee (a\ast 1).$ Implies that $a\ast 0+ a\ast 1=a\vee (a\ast 1)\implies a+(a\ast 1)=1.$
	\end{proof}
	\begin{theorem}
		The complemented elements of $A$ form a Boolean algebra.
		\begin{proof}
			Thus, we also have:$(a \ast 1) + a =1 = a' + a \Longrightarrow a \ast 1 \leq a'$ $ \Longrightarrow a \wedge (a \ast 1) \leq a \wedge a' = 0.$
			This leads to:$	a \ast (a \wedge (a \ast 1)) + (a \ast 1)= a \vee (a \ast 1)\Longrightarrow a \ast 0 + a \ast 1 = a \vee (a \ast 1)=a + (a \ast 1) = 1.$
			Thus, $a' = a \ast 1$. Since the set of all complemented elements forms a sublattice of $A$ having $0$ and $1$, it follows that $A$ is a Boolean algebra.
		\end{proof}
	\end{theorem}
	
	\begin{theorem}
		$A$ with unity is a Boolean algebra if for each $a \in A$, the mapping $x \mapsto a \ast x$ is an isometry.
	\end{theorem}
	\begin{proof}
		Since the mapping $x \mapsto 1 \ast x$ is an isomorphism, this implies \( a = 1 \ast x \) for some \( x \), hence \( a \geq 0 \). Also, \( a + (1 \ast a) = 1 \Longrightarrow a \leq 1 \). Thus, \( A = A^{+} \) with the greatest element \( 1 \). 
		
		Let \( a \in A \). Then:
		$a \vee (1 \ast a) = a \ast (1 \ast a) + a \wedge (1 \ast a)  = (a \ast 0) \ast (a \ast 1) + a \wedge (a \ast 1) = 0 \ast 1 + a \wedge (a \ast 1) = 1 + a \wedge (a \ast 1) = 1.$
		
		Now consider:$1 \ast (a \wedge (1 \ast a)) = 1 \wedge (a \wedge (1 \wedge a'))'=(1 \wedge a') \vee (1 \wedge (1 \wedge a')) = (1 \ast a) \vee (1 \ast (1 \ast a)) =1 = 1 \ast 0,$
		which implies \( a \wedge (1 \ast a) = 0 \). Thus, \( a \) is complemented. Hence, \( A \) is a Boolean algebra.
	\end{proof}
	
	\begin{theorem}
		If \( (A, +, \leq, 0) \) is a commutative monoid that is a chain such that:
		\begin{itemize}
			\item[1.] \( x \leq y \Rightarrow a + x \leq a + y \)
			\item[2.] \( a \ast (a \wedge b) + b = a \vee b \)
			\item[3.] The mappings \( x \mapsto a + x, a \vee x, a \wedge x, a \ast x \) are contractions with respect to \( \ast \)
		\end{itemize}
		then \( (A, +, \leq, 0, \ast) \) is an AL-monoid.
		\begin{proof}
			We will show that \( [a \ast (a \vee b)] \wedge [b \ast (a \wedge b)] = 0 \). Since \( A \) is a chain, we have either \( a = a \vee b \) or \( b = a \vee b \). 
			
			If \( b = a \vee b \), then:$
			(a \ast (a \vee b)) \wedge (b \ast (a \vee b)) = (a \ast a) \wedge (b \ast a) = 0 \wedge (b \ast a) = 0.
			$
			
			If \( a = a \vee b \), then:
			$(a \ast (a \vee b)) \wedge (b \ast (a \vee b)) = (a \ast b) \wedge 0 = 0.$
			
			Thus, \( A \) is an AL-monoid.
		\end{proof}
	\end{theorem}
	
\end{document}